\documentclass[twoside,leqno,10pt]{amsart}
\usepackage{amsfonts}
\usepackage{amsmath}
\usepackage{amscd}
\usepackage{amssymb}
\usepackage{amsthm}
\usepackage{amsrefs}
\usepackage{enumerate}
\usepackage{latexsym}
\usepackage{bbm}
\begin{document}
\newtheorem{theorem}{Theorem}
\newtheorem{lemma}{Lemma}
\newtheorem{prop}{Proposition}
\newtheorem{corollary}{Corollary}
\newtheorem{conjecture}{Conjecture}
\numberwithin{equation}{section}
\newcommand{\dif}{\mathrm{d}}
\newcommand{\intz}{\mathbb{Z}}
\newcommand{\ratq}{\mathbb{Q}}
\newcommand{\natn}{\mathbb{N}}
\newcommand{\comc}{\mathbb{C}}
\newcommand{\rear}{\mathbb{R}} 
\newcommand{\prip}{\mathbb{P}}
\newcommand{\uph}{\mathbb{H}}
\newcommand{\logl}{\mathcal{L}}

\title{\bf Gaps of Smallest Possible Order between Primes in an Arithmetic Progression}
\author{Roger C. Baker and Liangyi Zhao}
\date{\today}
\maketitle

\begin{abstract}
Let $t \geq 1$.  Suppose that $x$ is a sufficiently large real number and $q$ is a natural number with $q = x^{\theta}$,
\[ \theta \leq \frac{5}{12} - \eta, \; \prod_{p|q} p \leq \left( \log x \right)^C, \]
where $\eta$ and $C$ are positive constants.  Let $(a,q)=1$.  Then there are primes $p_1 < \cdots < p_t$ congruent to $a$ modulo $q$ in $( x/2, x]$ with
\[ p_t - p_1 \ll q \exp \left( Bt \right) \]
whenever $B > \frac{40}{9 - 20 \theta}$.
\end{abstract}



\section{Introduction}

Let $t \in \natn$ and $0 \leq \eta < 1$ be given.  Suppose that $x$ is a large positive real number, and that $q \in \natn$ and $(a,q) = 1$, $q \leq x^{1-\eta}$.  Set
\[ \mathcal{A} = \{ n \in (x/2, x] : n \equiv a \pmod{q} \}. \]
It may be conjectured that there are primes $p_1 < p_2 < \cdots <  p_t$ in $\mathcal{A}$ with
\begin{equation} \label{boundgapres}
p_t - p_1 \ll_t q.
\end{equation}
J. Maynard \cite{May} has recently refined the Goldston-Pintz-Y{\i}ldr{\i}m sieve to prove this in the case of $q=1$, showing that
\[ p_t -p_1 \ll t^3 \exp(4t) . \]

In this paper, we prove \eqref{boundgapres} for $q$'s in a restricted class.  Namely, we suppose that for some positive constant $\eta$ and $C$, we have $q = x^{\theta}$,
\begin{equation} \label{thetaccond}
\theta \leq \frac{5}{12} - \eta, \; \prod_{p|q} p \leq \left( \log x \right)^C .
\end{equation}
Before stating our results, we address the question of where the hypotheses in \eqref{thetaccond} come from.  There are many striking results on the existence of primes in arithmetic progressions.  D. R. Heath-Brown \cite{HB3} has shown that for any $q$ and $(a,q)=1$, the least prime $p(q,a)$ congruent to $a \pmod{q}$ satisfies
\[ p(q,a) \ll q^{5.5} . \]
G. Harman \cite{Harman} has shown, subject to a weak hypothesis on the zeros of $L(s,\chi)$ for characters $\chi \pmod{q}$, that
\[ \pi (x; q,a) \gg \frac{x}{\varphi(q) \log x} \; \mbox{for} \; q < x^{0.4736} \; \mbox{and} \; (a,q)=1. \]
(As usual,
\[ \pi(x; q,a) = \sum_{\substack{ p \leq x \\ p \equiv a \bmod{q}}} 1 \; \; \; \mbox{and} \; \; \; \psi(x;q,a) = \sum_{\substack{n \leq x \\ n \equiv a \bmod{q}}} \Lambda(n) . ) \]

However, in the present paper we need a result of Bombieri-Vinogradov type in order to employ Maynard's method.  For a positive constant $b$, let
\[ E_b (x,q) = \sum_{\substack{d \leq x^b \\ (d,q)=1}} \max_{(a,qd)=1} \left| \psi (x; qd, a) - \frac{x}{\varphi(qd)} \right| . \]
We shall require a large logarithm power saving over the trivial bound for $E_b(x,q)$.  Elliott \cite{El2} achieves this for
\[ q \leq x^{1/3} \exp \left( - (\log \logl)^3 \right) , \; q = c^n \]
where $c$ is a given natural number.  Here and below, $\logl : = \log x$.  We weaken these restrictions, replacing them by \eqref{thetaccond}.  To do better, we would respectively need improvements of a Huxley-Jutila zero density theorem \cites{Hux, Jut} and Siegel's theorem \cite[Chapter 21]{HD}. \newline

One of our key tools, the bounds of H. Iwaniec \cite{HI6} on $L$-functions, has been improved for certain ranges by M.-C. Chang \cite{Cha}.  This would not help us with $E_b(x,q)$, but see a recent paper by Banks, Freiberg and Maynard \cite{BanFreMay} for a closely related sum which requires Chang's work. \newline

In the sequel, let $\varepsilon $ denote a positive constant sufficiently small in terms of $\eta$ and $C$. 

\begin{theorem} \label{bomvino}
Suppose that \eqref{thetaccond} holds.  Let
\[ L (\theta) = \left\{ \begin{array} {ll} 1/2 - \theta - \varepsilon & \mbox{if} \; \theta < 2/5 - \varepsilon , \\ &  \\ 9/20 - \theta - \varepsilon & \mbox{otherwise}. \end{array} \right. \]
Then for $A>0$,
\[ E_{L(\theta)} (x,q) \ll \frac{x}{\varphi(q) \logl^A} . \]
The implied constant depends on $C$, $\theta$, $\varepsilon$ and $A$.
\end{theorem}

\begin{theorem} \label{maintheo}
Suppose that \eqref{thetaccond} holds.  For sufficiently large $x$, there are primes $p_1 < \cdots < p_t$ in $(x/2, x]$ congruent to $a \pmod{q}$ with
\[ p_t - p_1 \ll q \exp \left( \frac{2t}{L(\theta)} \right) . \]
The implied constant depends on $t$, $\eta$, $C$ and $\varepsilon$.
\end{theorem}

For completeness, we also include an analog of the Barban-Davenport-Halberstam theorem (\cite[Chapter 29]{HD}).
\begin{theorem} \label{bardavhal}
Suppose that \eqref{thetaccond} holds.  Let $A>0$.  Then we have
\[ \sum_{\substack{d \leq Q/q \\ (d,q)=1}} \sum_{\substack{a=1 \\ (a,qd)=1}}^{qd} \left( \psi(x; qd,a) - \frac{x}{\varphi(qd)} \right)^2 \ll \frac{xQ\logl}{\varphi(q)} \]
whenever $x \mathcal{L}^{-A} \leq Q \leq x$.  The implied constant depends on $\eta$, $C$ and $A$.
\end{theorem}

In what follows, we count the constant function $1$ as a primitive character.

\section{Preliminary Lemmas}

Unless otherwise stated, implied constants depend on $C$, $\eta$, $\varepsilon$ and $A$ (if present). \newline

For a Dirichlet character $\chi$, we use $\hat{\chi}$ to denote the primitive character that induces $\chi$.  Moreover, let
\[ \sideset{}{^{\prime}}\sum_{\chi \bmod{r}} \; \; \; \; \mbox{and} \; \; \; \; \sideset{}{^{\star}}\sum_{\chi \bmod{r}} \]
stand for, respectively, a sum restricted to nonprincipal characters modulo $r$ and a sum restricted to primitive nonprincipal characters modulo $r$.  As usual, let
\[ \psi(x , \chi) = \sum_{n \leq x} \Lambda(n) \chi(n) . \]

For a character $\chi \pmod{qd}$ where $(d,q)=1$, the conductor of $\chi$ takes the form $q_1(\chi) d_1(\chi)$ where
\[ q_1 (\chi) | q , \; d_1(\chi) | d . \]

\begin{lemma} \label{charredu}
\begin{enumerate}[(i)]
\item We have, for $r < x$,
\[ \max_{(a,r)=1} \left| \psi(x; r,a) - \frac{x}{\varphi(r)} \right| \ll \frac{1}{\varphi(r)} \sideset{}{^{\prime}}\sum_{\chi \bmod{r}} | \psi(x, \chi) | + \frac{x}{\varphi(r) \mathcal{L}^A} . \]
\item For each of the characters $\chi$ in the above, we have
\[ | \psi(x,\chi) | - | \psi(x, \hat{\chi}) | \ll \mathcal{L}^2 . \]
\item We have
\[ \sum_{\substack{a=1 \\ (a,r)=1}}^r \left( \psi(x ; r,a) - \frac{x}{\varphi(r)} \right)^2 \ll \frac{1}{\varphi(r)} \sideset{}{^{\prime}}\sum_{\chi \bmod{r}} | \psi(x,\chi) |^2 + \frac{x^2}{\varphi(r) \mathcal{L}^A} . \]
\end{enumerate}
\end{lemma}

\begin{proof}
These are standard results.  See, for example, pp. 162-163 and 169-170 in \cite{HD}.
\end{proof}

\begin{lemma} \label{charprimredu}
\begin{enumerate}[(i)]

\item For any natural number $r$ and any complex-valued function $F$ defined on Dirichlet characters, we have
\begin{equation} \label{chartrans1}
\sideset{}{^{\prime}}\sum_{\chi \bmod{r}} F ( \hat{\chi} ) = \sum_{r_1 | r}  \ \ \sideset{}{^{\star}}\sum_{\chi_1 \bmod{r_1}} F ( \chi_1) .
\end{equation}

\item Let $H>0$.  Suppose further that $F \geq 0$, that $F(\hat{\chi}) =0$ for $d_1(\chi)  \leq H$, and $qD < x$.  There exists $D_1 \in (H, D]$ such that
\begin{equation} \label{chartrans2}
\sum_{\substack{d \leq D \\ (d,q)=1}} \ \ \sideset{}{^{\prime}}\sum_{\chi \bmod{qd}} F ( \hat{\chi} ) \ll \frac{\mathcal{L} D}{D_1} \sum_{q_1 | q} \sum_{\substack{D_1 < d \leq 2D_1 \\ (d,q)=1}} \ \sideset{}{^{\star}}\sum_{\chi \bmod{q_1d_1}} F(\chi) . 
\end{equation}

\end{enumerate}
\end{lemma}

\begin{proof}
The equation \eqref{chartrans1} is immediate from allocating the conductors of $\hat{\chi}$ into classes corresponding to divisors of $r$.  For \eqref{chartrans2}, the left-hand side is
\begin{equation*}
\begin{split}
 \sum_{\substack{d \leq D \\ (d,q)=1}} \ \sum_{\substack{q_1 | q, d_1 | d \\ d_1 > H}} \ \ \sideset{}{^{\star}}\sum_{\chi_1 \bmod{q_1d_1}} F(\chi_1) & = \sum_{q_1 |q} \sum_{\substack{H< d_1 \leq D \\ (d_1,q)=1}} \left( \sum_{\substack{d \leq D \\ (d,q)=1 \\ d \equiv 0 \bmod{d_1}}} 1 \right) \sideset{}{^{\star}}\sum_{\chi_1 \bmod{q_1d_1}} F(\chi_1) \\
 & \ll \mathcal{L} \frac{D_1}{D} \sum_{q_1 | q} \sum_{\substack{ D_1 < d_1 \leq 2D_1 \\ (d_1, q) = 1}} \ \sideset{}{^{\star}}\sum_{\chi_1 \bmod{q_1d_1}} F(\chi_1) 
 \end{split}
 \end{equation*}
 for some $D_1$, $H \leq D_1 \leq D$, by splitting the range of $d_1$ into dyadic intervals.  This completes the proof.
\end{proof}

As an example of the last lemma, let $\varphi^* (r)$ denote the number of primitive characters modulo $r$.  Then
\begin{equation} \label{primcharnum}
\sum_{r_1 | r} \varphi^* (r_1) = \varphi(r) .
\end{equation}

\begin{lemma} \label{psitoR}
Let $L = L(\theta )$ as in Theorem~\ref{bomvino} and
\[ R(x ; r, a) = \sum_{\substack{ n \leq x \\ n \equiv a \bmod{r}}} \Lambda (n) \log \frac{x}{n} . \]
Suppose that for $D \ll x^L$ and some $A>0$,
\begin{equation} \label{transcond}
\sum_{\substack{D < d \leq 2D \\  (d,q)=1}} \max_{(a,qd)=1} \left| R ( x; qd, a) - \frac{x}{\varphi(qd)} \right| \ll \frac{x}{\varphi(q) \mathcal{L}^{2A+1}}.
\end{equation}
Then for $D \ll x^L$,
\[ \sum_{\substack{D < d \leq 2D \\  (d,q)=1}} \max_{(a,qd)=1} \left| \psi ( x; qd, a) - \frac{x}{\varphi(qd)} \right| \ll \frac{x}{\varphi(q) \mathcal{L}^A}. \]
\end{lemma}

\begin{proof}
We start with the identity
\[ R (x; r,a) = \int\limits_1^x \psi (y; r,a) \frac{\dif y}{y} . \]
This, together with the fact that $\psi(y;r,a)$ is nondecreasing in $y$, gives that for all $\lambda >0$,
\begin{equation*}
\begin{split}
 \frac{R(x;r,a) - R(xe^{-\lambda}; r,a)}{\lambda} & = \frac{1}{\lambda} \int\limits_{e^{-\lambda}x}^x \psi(y;r,a) \frac{\dif y}{y} \\
 & \leq \psi(x;r,a) \leq \frac{1}{\lambda} \int\limits_x^{e^{\lambda} x} \psi(y;r,a) \frac{\dif y}{y} = \frac{ R(e^{\lambda}x; r,a) - R(x;r,a)}{\lambda} .
 \end{split}
 \end{equation*}
 This leads to
 \[ \psi(x;r,a) - \frac{x}{\varphi(r)} \leq \frac{R(e^{\lambda} x; r,a) - e^{\lambda} x / \varphi(r)}{\lambda} - \frac{R(x;r,a) - x/\varphi(r)}{\lambda} + \left( \frac{e^{\lambda}-1}{\lambda} - 1 \right) \frac{x}{\varphi(r)} \]
and
\[ \psi(x;r,a) - \frac{x}{\varphi(r)} \geq \frac{R(x;r,a) - x/\varphi(r)}{\lambda} - \frac{R(e^{-\lambda} x ; r,a) - e^{-\lambda}x/\varphi(r)}{\lambda} + \left( \frac{1-e^{-\lambda}}{\lambda} - 1 \right) \frac{x}{\varphi(r)} . \]
Take $\lambda = \mathcal{L}^{-A-1}$ so that
\[ \frac{e^{\lambda}-1}{\lambda} -1 \ll \mathcal{L}^{-A-1} \; \; \; \mbox{and} \; \; \; \frac{1-e^{-\lambda}}{\lambda} -1 \ll \mathcal{L}^{-A-1} . \]
We get, taking $D \ll x^L$, $r=qd$ and summing over $d \in (D, 2D]$, there is $\mu \in \{ 1, 0, -1 \}$ for which
\begin{equation*}
\begin{split}
 \sum_{\substack{D < d \leq 2D \\ (d,q)=1}} & \max_{(a,qd)=1} \left| \psi(x;qd,a) - \frac{x}{\varphi(qd)} \right| \\
 & \ll \mathcal{L}^{A+1} \sum_{\substack{D < d \leq 2D \\ (d,q)=1}} \max_{(a,qd)=1} \left| R(e^{\mu}x;qd,a) - \frac{e^{\mu}x}{\varphi(qd)} \right| + \frac{x \mathcal{L}^{-A-1}}{\varphi(q)} \sum_{1 \leq d \leq 2D} \frac{1}{\varphi(d)} \ll \frac{x}{\varphi(q) \mathcal{L}^A} ,
 \end{split}
 \end{equation*}
 using \eqref{transcond} with $e^{\mu}x$ in the place of $x$.
\end{proof}

In the following lemma, let $\beta + i \gamma$ denote a zero of any of the Dirichlet $L$-functions $L(s,\chi)$ with $\chi$ a non-principal character modulo $r$.
\begin{lemma} \label{elllem}
Let $r < x$.  Then
\[ \sideset{}{^{\prime}}\sum_{\chi \bmod{r}} \left| \psi (x, \chi) \right| \ll \sideset{}{^{\prime}}\sum_{\chi \bmod{r}} \sum_{\substack{\beta \geq 1/2 \\ |\gamma| < x^{1/2}}} \frac{x^{\beta} \mathcal{L}^{A+1}}{| \beta + i \gamma|^2} + x^{1/2} r \mathcal{L}^2  + \frac{x}{\mathcal{L}^A} . \]
\end{lemma}

\begin{proof}
This is a very slight variant of a result established by Elliott \cite[pp. 248-249]{El2}.
\end{proof}

Let $N(\sigma, T, \chi)$ denote the number of zeros of $L(s,\chi)$ in the rectangle $[ \sigma, 1) \times [-T,T]$.  We shall need the following zero density result.

\begin{lemma} \label{zerodensity}
We have, for $T \geq 1$, $1/2 \leq \sigma < 1$
\[ \sideset{}{^{\prime}}\sum_{\chi \bmod{r}} N (\sigma, T, \chi) \ll (rT)^{(12/5+\varepsilon)(1-\sigma)}. \]
\end{lemma}

\begin{proof}
This is obtained by combining the results of M. N. Huxley \cite{Hux} and M. Jutila \cite{Jut}.
\end{proof}

\begin{lemma} \label{larsie}
Let $a_n$ ($n =1, \cdots, N$) be complex numbers and
\begin{equation} \label{Tchidef}
 T(\chi) = \sum_{n=1}^N a_n \chi(n) .
 \end{equation}
For any natural numbers $r$ and $D$, we have
\[ \sum_{r_1 | r} \sum_{\substack{d \leq D \\ (d,r)=1}} \frac{r_1d}{\varphi(r_1d)} \ \sideset{}{^{\star}}\sum_{\chi \bmod{r_1d}} \left| T ( \chi ) \right|^2 \ll (N + rD^2 ) \sum_{n=1}^N |a_n|^2 . \]
\end{lemma}

\begin{proof}
This is a variant of Lemma 6.5 in \cite{El}.  Set
\[ S(x) = \sum_{n=1}^N a_n e(nx) , \]
where $e(z) = \exp (2 \pi i z)$.  Let 
\[ \mathcal{S} = \left\{ \frac{j}{dr_1} \in \ratq : 1 \leq j \leq dr_1, (j, dr_1 ) =1, d \leq D, (d, r)=1, r_1 | r \right\} . \]
It is easy to see that
\[ |s - s'| \geq \frac{1}{rD^2} \]
for all distinct $s$ and $s'$ in $\mathcal{S}$.  From the classical large sieve inequality (see \cite[Chapter 27]{HD}), we get
\[ \sum_{s \in \mathcal{S}} \left| S(s) \right|^2 \ll \left( N + rD^2 \right) \sum_{n=1}^N |a_n|^2 . \]
Now by standard techniques that relate multiplicative characters to additive ones (see (10) on page 160 of \cite{HD}), we get
\begin{equation} \label{multtoadd}
\sideset{}{^{\star}}\sum_{\chi \bmod{r_1d}} \frac{r_1d}{\varphi(r_1d)} \left| T(\chi) \right|^2 \leq \sum_{\substack{j=1 \\ (j,r_1d)=1}}^{r_1d} \left| S \left( \frac{j}{r_1d} \right) \right|^2 .
\end{equation}
Now the lemma follows by summing over pairs of $r_1$ and $d$ with $r_1 | r$ and $d \leq D$ with $(d,r)=1$ in \eqref{multtoadd}.
\end{proof}

\begin{lemma} \label{huxlem}
Let $N \leq x$, $qD \leq x$ and $\mathcal{U}$ be a set of non-principal characters to moduli $q_1d$ with $d \leq D$, $(d,q)=1$ and $q_1 | q$.  Suppose that, with $T(\chi)$ as in \eqref{Tchidef},
\[ | T(\chi) | \geq V > 0 \]
whenever $\chi \in \mathcal{U}$ and that $G = \sum_{n=1}^N |a_n|^2$.  Then
\[ \# \mathcal{U} \ll x^{\varepsilon/20} \left( GV^{-2} N + G^3 V^{-6} Nq D^2 \right). \]
\end{lemma}

\begin{proof}
We first suppose that
\[ V > G^{1/2} N^{1/4} x^{\varepsilon/80}. \]
The contribution to $\# \mathcal{U}$ from a fixed $q_1 | q$ is
\[ \ll x^{\varepsilon/40} \left( GV^{-2} N + G^3 V^{-6} Nq_1 D^2 \right) \]
by virtue of \cite[Theorem 1]{Hux}.  The lemma follows on summing over $q_1$ with $q_1 | q$. \newline

Now suppose that
\[ V \leq G^{1/2} N^{1/4} x^{\varepsilon/80}. \]
From Lemma~\ref{larsie},
\[ \# \mathcal{U} \ll G(N+qD^2)V^{-2} \ll GNV^{-2} + G^3 V^{-6} NqD^2 x^{\varepsilon/20} . \]
\end{proof}

\begin{lemma} \label{L4thmom}
For $r \geq 3$ and $T \geq 1$,
\[ \sideset{}{^{\star}}\sum_{\chi \bmod{r}} \int\limits_0^T \left| L \left( \frac{1}{2} + it, \chi \right) \right|^4 \dif t \ll \varphi^* (r) T (\log rT)^4 . \]
\end{lemma}

\begin{proof}
See \cite{HBB} for a more precise form of this result.
\end{proof}

\begin{lemma} \label{Diripoly4powerest}
Let $qD < x$, $N \leq x$, $|t| \leq x^2$ and
\[ N (s, \chi) = \sum_{N < n \leq N'} \chi(n) n^{-s} \]
where $N$ and $N'$ are natural numbers with $N' \leq 2N$.  Then
\[ \sum_{q_1| q} \sum_{\substack{d \leq D \\ (d,q)=1}} \sideset{}{^{\star}}\sum_{\chi \bmod{q_1d}} \left| N \left( \frac{1}{2} + it, \chi \right) \right|^4 \ll \varphi(q) D^2 \mathcal{L}^5 \left( 1 + |t| \right) . \]
\end{lemma}

\begin{proof}
Using Perron's formula (\cite[Lemma 3.12]{ET}), we see that
\begin{equation} \label{applyperron}
 N \left( \frac{1}{2} + it, \chi \right) = \frac{1}{2\pi i} \int\limits_{1-ix^2}^{1+ix^2} L \left( \frac{1}{2} + it + w, \chi \right) \left( \frac{(N'+1/2)^w-(N+1/2)^w}{w} \right) \dif w +O(1) .
 \end{equation}
From the work of Heath-Brown \cite{HB4}, we have $L (\sigma + it, \xi) \ll (r(|t|+1))^{3/16+\varepsilon}$ for $\sigma \geq 1/2$, where $\xi$ is a character modulo $r$.  Using this bound, we can move the line of integration in \eqref{applyperron} to $[-ix^2, ix^2]$ at the cost of an error of size $O(1)$.  This follows from the observation that on the horizontal line segments from $\pm ix^2$ to $1\pm ix^2$, the integrand on the right-hand side of \eqref{applyperron} is
\[ \ll (qD)^{3/16+\varepsilon} N x^{-2+3/8+\varepsilon} \ll x^{9/16-1+2\varepsilon} \ll 1. \]
Now by a splitting-up argument, it suffices to show that for $1 \leq T \leq x^2$ that
\begin{equation} \label{afterred}
\sum_{q_1 | q} \sum_{\substack{ d \leq D \\ (d,q)=1}} \sideset{}{^{\star}}\sum_{\chi \bmod{dq_1}} \left( \frac{1}{T} \int\limits_{T-1}^{2T} \left| L \left( \frac{1}{2} + it + iu , \chi \right) \right| \dif u \right)^4 \ll \varphi(q) D^2 \mathcal{L}^4 (1+|t|) . 
\end{equation}

By H\"older's inequality,
\[ \left( \frac{1}{T} \int\limits_{T-1}^{2T} \left| L \left( \frac{1}{2} + it + iu , \chi \right) \right| \dif u \right)^4 \ll \frac{1}{T} \int\limits_{T-1}^{2T} \left| L \left( \frac{1}{2} + it + iu , \chi \right) \right|^4 \dif u . \]

Recalling Lemma~\ref{L4thmom} and \eqref{primcharnum}, the left-hand side of \eqref{afterred} is
\begin{equation*}
\begin{split}
 \ll \frac{1}{T} & \sum_{q_1 | q} \sum_{\substack{ d \leq D \\ (d,q)=1}} \ \sideset{}{^{\star}}\sum_{\chi \bmod{q_1d}} \ \int\limits_{T-1+t}^{2T+t} \left| L \left( \frac{1}{2} + iv , \chi \right) \right|^4 \dif v \\
 & \ll \frac{1}{T} \sum_{q_1 | q} \sum_{\substack{ d \leq D \\ (d,q)=1}} \varphi^* (q_1d) \mathcal{L}^4 (T + |t|) =  \frac{\varphi(q)}{T} \sum_{\substack{ d \leq D \\ (d,q)=1}} \varphi^*(d) \mathcal{L}^4 (T+|t|).
 \end{split}
 \end{equation*}
Now the lemma follows at once from this.
\end{proof}

Next, we have the Heath-Brown decomposition of the von Mangoldt function.

\begin{lemma} \label{HBID}
Let $f(n)$ be an arbitrary complex-valued function and $k \in \natn$.  We can decompose the sum
\[ \sum_{n\leq x} \Lambda(n) f(n) \]
into $O(\mathcal{L}^{2k})$ sums of the form
\begin{equation}
\sum_{\substack{n_i \in [N_i, 2N_i) \\ n_1 \cdots  n_{2k} \leq x}} \log n_1 \mu(n_{k+1}) \cdots \mu(n_{2k}) f(n_1 \cdots n_{2k})
\end{equation}
in which $N_i \geq 1$, $\prod_i N_i < x$ and $2N_i \leq x^{1/k}$ if $i > k$.
\end{lemma}

\begin{proof}
This is from \cite{HB2}.
\end{proof}

\begin{lemma} \label{notzerolem}
Suppose that \eqref{thetaccond} holds.  Let $x \geq C_3(C)$ and $q \leq x$.  For every primitive character $\chi \pmod{q}$, we have
\[ L(s,\chi) \neq 0, \; \mbox{if} \;  \left| \Im s \right| \leq x, \; \Re s > 1 - \logl^{-4/5} . \]
\end{lemma}

\begin{proof}
This is well-known for $q=1$ and $\chi =1$.  Suppose that $q >1$.  Let
\[ d = \prod_{p | q} p, \; l = \log q (x+3) , \; \theta = \frac{1}{4 \cdot 10^4 \left(  \log d + ( l \log 2l )^{3/4} \right)} \geq \logl^{-4/5} . \]
According to \cite[Theorem 2]{HI6}, there is at most one primitive character $\chi \pmod{q}$ such that there is $\rho$ with
\[ \Re \rho > 1- \theta, \; \left| \Im \rho \right| \leq x, \; \mbox{and} \; L (\rho, \chi) = 0. \]
We suppose if possible that $\chi$ exists.  In this case, from \cite[Theorem 2]{HI6}, $\chi$ is real and $\rho$ is real.  From \cite[Page 40]{HD}, $q/(q,8)$ is squarefree, so that $d \geq q/8$; and from \cite[page 126]{HD},
\[ \rho \leq 1- \frac{C_4(C)}{q^{1/(2C)}} \leq 1 - \frac{C_4(C)8^{1/(2C)}}{d^{1/(2C)}} \leq 1 - \frac{C_4(C) 8^{1/(2C)}}{\logl^{1/2}} . \]
Therefore,
\[ \frac{\logl^{1/2}}{C_4(C) 8^{1/(2C)}} \geq \logl^{4/5} , \]
which is absurd.  This completes the proof of the lemma.
\end{proof}

\section{Proof of Theorem~\ref{bomvino}}

Let $(d,q)=1$.  Given a character $\chi \pmod{qd}$ induced by $\hat{\chi}$, let $\chi^{\dagger}$ denote the character $\pmod{qd_1(\chi)}$ induced by $\hat{\chi}$. \newline

Note that $\chi^{\dagger}$ shares with $\hat{\chi}$ the property
\[ \left| \psi(y, \chi^{\dagger}) \right| = \left| \psi (y, \chi) \right| + O \left( \logl^2 \right) , \; (1 \leq y \leq x) . \]
This is a consequence of Lemma~\ref{charredu} (ii). \newline

Recalling Lemma~\ref{psitoR}, in order to prove Theorem~\ref{bomvino}, it remains to show that
\[ \sum_{\substack{D \leq d \leq 2D \\ (d,q)=1}} \left| R(x; qd, a(d)) - \frac{x}{\varphi(qd)} \right| \ll \frac{x}{\varphi(q) \logl^{2A+3}} \]
whenever $1 \leq D \ll X^{L(\theta)}$, for any sequence $a(d)$ with $(a(d), dq)=1$. \newline

Now
\begin{equation} \label{replaceint}
R(x;qd,a(d)) - \frac{x}{\varphi(qd)} = \frac{1}{\varphi(qd)} \sideset{}{^{\prime}}\sum_{\chi \bmod{qd}} \bar{\chi} (a(d)) \int\limits_1^x \psi(y,\chi) \frac{\dif y}{y} + O \left( \frac{x}{\varphi(qd) \logl^{2A+4}} \right) .
\end{equation}
By replacing $\psi(y,\chi)$ by $\psi(y, \hat{\chi})$ or $\psi(y, \chi^{\dagger})$ in \eqref{replaceint}, we incur an error of size
\[ \ll \logl^3 \ll \frac{x}{\varphi(qd) \logl^{2A+4}} . \]
Therefore, it suffices to show for some absolute positive constant $C_4$ that
\begin{equation} \label{chidagineq}
\sum_{\substack{D \leq d \leq 2D \\ (d,q)=1}} \sideset{}{^{\prime}} \sum_{\substack{\chi \bmod{qd} \\ d_1(\chi) < \logl^{2A+C_4} }} \left| \int\limits_1^x \psi(y, \chi^{\dagger}) \frac{\dif y}{y} \right| \ll \frac{xD}{\logl^{2A+4}}  
\end{equation}
and that
\begin{equation} \label{chihatineq}
\sum_{\substack{D \leq d \leq 2D \\ (d,q)=1}} \sideset{}{^{\prime}} \sum_{\substack{\chi \bmod{qd} \\ d_1(\chi) \geq \logl^{2A+C_4} }} \left| \int\limits_1^x \psi(y, \hat{\chi}) \frac{\dif y}{y} \right| \ll \frac{xD}{\logl^{2A+4}}  
\end{equation}

We begin with \eqref{chidagineq}.  We observe that if $\chi^{\dagger} \pmod{qd_1}$ is given and $d$ is an integer divisible by $d_1$, then $\chi^{\dagger}$ determines $\chi \pmod{qd}$, since it is easy to see that
\[ \chi(n) = \left\{ \begin{array}{cl} \chi^{\dagger} (n) & \mbox{if} \; (n,qd)=1 \\ 0 & \mbox{if} \; (n,qd)>1. \end{array} \right. \]
Hence
\begin{equation*}
\begin{split} \sum_{\substack{D \leq d \leq 2D \\ (d,q)=1}} \sideset{}{^{\prime}} \sum_{\substack{\chi \bmod{qd} \\ d_1(\chi) < \logl^{2A+C_4} }} \left| \int\limits_1^x \psi(y, \chi^{\dagger}) \frac{\dif y}{y} \right| & \leq \sum_{\substack{d_1 < \logl^{2A+C_4} \\ (d_1,q)=1}} \ \sideset{}{^{\prime}} \sum_{\chi_1\bmod{qd_1}} \sum_{\substack{D \leq d < 2D \\ (d,q)=1 , d_1|d}} \left| \int\limits_1^x \psi(y, \chi_1) \frac{\dif y}{y} \right| \\
& \ll \frac{\logl^2D}{D_1} \sum_{\substack{D_1 < d_1 \leq 2D_1 \\ (d_1,q)=1}}  \ \sideset{}{^{\prime}} \sum_{\chi_1\bmod{qd_1}} \left| \psi(y_1, \chi_1) \right| ,
\end{split}
\end{equation*}
for some $D_1 \in [1, \logl^{2A+C_4} )$ and some $y_1 = y_1(\chi)$, $1 \leq y_1 \leq x$.  Thus we must show that
\[ \sum_{D_1 < d_1 \leq 2D_1} \ \sideset{}{^{\prime}} \sum_{\chi_1\bmod{qd_1}} \left| \psi(y_1, \chi_1) \right| \ll \frac{xD_1}{\logl^{2A+6}} . \]

Now $qD_1 < x^{1/2-\varepsilon}$.  In view of Lemma~\ref{elllem}, with $r$, $A$ replaced by $qd_1$, $2A+6$, it suffices to show that ($\beta + i \gamma$ denoting a zero of $L(s,\chi_1)$)
\[ \sideset{}{^{\prime}} \sum_{\chi_1\bmod{qd_1}} \sum_{\substack{\beta >1/2 \\ |\gamma| < x^{1/2}}} \frac{x^{\beta}\logl^{2A+7}}{|\beta+i \gamma|^2} \ll \frac{x}{\logl^{2A+6}}  \]
for each $d_1 < \logl^{2A+C_4}$.  Here, the left-hand side is
\begin{equation} \label{zerobound}
\ll \logl^{2A+8} x^{\sigma} \ \sideset{}{^{\prime}} \sum_{\chi_1\bmod{qd_1}} \sum_{\substack{\sigma \leq \beta \leq \sigma + \logl^{-1} \\ |\gamma|<x^{1/2}}} \frac{1}{|\beta + i \gamma|^2}
\end{equation}
for some $\sigma$, $1/2 \leq \sigma < 1$.  From Lemma~\ref{notzerolem}, applied to $\hat{\chi}_1$, with $C+2A+C_4$ in place of $C$, the sum in \eqref{zerobound} is empty if
\[ \sigma \geq 1 - \logl^{-4/5} . \]
Suppose now that $ \sigma < 1- \logl^{-4/5}$.  It suffices to show that
\[ S:=  \sideset{}{^{\prime}} \sum_{\chi_1\bmod{qd_1}} \sum_{\substack{\sigma \leq \beta < \sigma + \logl^{-1} \\ |\gamma|< x^{1/2}}} \frac{1}{|\beta + i \gamma|^2} \ll \frac{x^{1-\sigma}}{\logl^{4A+14}} . \]
Using Lemma~\ref{zerodensity},
\[ S \ll  \sideset{}{^{\prime}} \sum_{\chi_1\bmod{qd_1}} \sum_{j \geq 0} 2^{-2j} N (\sigma, 2^{j+1}, \chi_1) \ll \sum_{j \geq 0} 2^{-j/2} (qd_1)^{(12/5+\varepsilon)(1-\sigma)} \ll \logl^{3A+2C_4} x^{(12/5+\varepsilon)(5/12-\eta)(1-\sigma)} . \]
Therefore,
\[ Sx^{-(1-\sigma)} \logl^{4A+14} \ll x^{-(12\eta/5-\varepsilon)(1-\sigma)} \logl^{7A+2C_4+14} \ll \exp \left( - \left( \frac{12\eta}{5}-\varepsilon \right) \logl^{1/5} \right) \logl^{7A+2C_4+14}  \ll 1 . \]
This completes the proof of \eqref{chidagineq}. \newline

For \eqref{chihatineq}, we apply (ii) of Lemma~\ref{charprimredu}.  We need only show for $\logl^{2A+C_4} \leq D_1 \leq D$ that
\[ S(D_1) : = \sum_{q_1 | q} \ \sum_{\substack{D_1 < d \leq 2D_1 \\ (d, q) =1}} \ \sideset{}{^{\star}}\sum_{\chi \bmod{q_1d}} \left| \int\limits_1^x \psi (y, \chi) \frac{\dif y}{y} \right| \ll \frac{xD_1}{\mathcal{L}^{2A+5}} . \]

For brevity, we write $\sum^{\dagger}$ in place of
\[ \sum_{q_1 | q} \ \sum_{\substack{D_1 < d \leq 2D_1 \\ (d, q) =1}} \ \sideset{}{^{\star}}\sum_{\chi \bmod{q_1d}}  \; \; . \]

Recasting the absolute value signs as coefficients, we have
\[ S(D_1) = \sideset{}{^\dagger} \sum b(\chi) \int\limits_1^x \psi(y, \chi) \frac{\dif y}{y} = \sideset{}{^\dagger} \sum b(\chi) \sum_{n \leq x} \Lambda(n) \chi (n) \log \frac{x}{n} . \]
Now applying Lemma~\ref{HBID} with $k=7$ and
\[ f(n) = \sideset{}{^\dagger} \sum b(\chi) \chi(n) \log \frac{x}{n} , \]
we see that it suffices to show for each tuple $N_1$, $\cdots$, $N_{14}$ that
\[ \sideset{}{^\dagger} \sum b(\chi) \sum_{\substack{n_i \in (N_i, 2N_i] \\ n_1 \cdots n_{14} \leq x}} a_1 (n_1) \cdots a_{14} (n_{14}) \chi (n_1 \cdots n_{14}) \log \frac{x}{n_1 \cdots n_{14}} \ll \frac{xD_1}{\mathcal{L}^{2A+19}} . \]
Here the coefficients $a_j(n_j)$ are those resulting from the application of Lemma~\ref{HBID}; that is,
\[ a_1 (n) = \log n, \; a_j (n) = 1 \; \mbox{for} \; 2 \leq j \leq 7  \; \mbox{and} \; a_j(n) = \mu(n) \; \mbox{for} \; 8 \leq j \leq 14. \]
Using the formula
\[ \int\limits_{1/2-i \infty}^{1/2+i \infty} y^s \frac{\dif s}{s^2} = \left\{ \begin{array}{cl} \log y, & \mbox{if} \; y > 1 \\ 0, & \mbox{if} \; 0 < y \leq 1  \end{array} \right. \]
(cf. \cite[p. 143]{MonVau}), we need to show that
\[ \sideset{}{^\dagger} \sum b(\chi) \int\limits_{1/2-i \infty}^{1/2 + i \infty} \sum_{\substack{n_i \in (N_i, 2N_i] \\ n_1 \cdots n_{14} \leq x}} \frac{a_1 (n_1) \cdots a_{14} (n_{14}) \chi (n_1 \cdots n_{14})}{(n_1 \cdots n_{14})^s} \frac{x^s \dif s}{s^2} \ll \frac{x D_1}{\mathcal{L}^{2A+19}} . \]
Now the condition $n_1 \cdots n_{14} \leq x$ can be removed, since the integral vanishes otherwise.  We also use a trivial estimate to discard the part of the integral with $| \Im s | > x^2$.  Thus our task is further reduced to showing that
\[  \sideset{}{^\dagger} \sum b(\chi) \int\limits_{-x^2}^{x^2} N_1 \left( \frac{1}{2} + it, \chi \right) \cdots N_{14} \left( \frac{1}{2} + it, \chi \right) \frac{x^{1/2+it}}{(1/2 + it)^2} \dif t \ll \frac{xD_1}{\mathcal{L}^{2A+19}} , \]
where
\[ N_j (s, \chi) = \sum_{N_j < n \leq 2N_j} \frac{a_j(n) \chi(n)}{n^s} . \]
To this end, it suffices to prove that
\begin{equation} \label{redugoal}
\sideset{}{^\dagger} \sum \left| N_1 \left( \frac{1}{2} + it, \chi \right) \cdots N_{14} \left( \frac{1}{2} + it, \chi \right)  \right| \ll \frac{x^{1/2} D_1 (1+|t|)}{\mathcal{L}^{2A+20}} 
\end{equation}
for $|t| \leq x^2$.  It is convenient to recall here that $qD_1 \ll x^{1/2-\varepsilon}$ for all $\theta$ and $qD_1 \ll x^{9/20-\varepsilon}$ for $\theta \geq 2/5-\varepsilon$. \newline

Let us write $x_0 = \prod_{i=1}^{14} N_i$ and $N_i = x_0^{\alpha_i}$ so that $\alpha_i \geq 0$, $\alpha_1 + \cdots + \alpha_{14} =1$ and $x_0 \leq x$. \newline

For a Dirichlet polynomial
\[ N(s) = \sum_{N < n \leq zN} a_n \chi(n) n^{-s} \]
for some constant $z >1$, we use the abbreviation, for $p > 1$,
\[ \| N \|_p = \left( \sideset{}{^\dagger} \sum \left| N \left( \frac{1}{2} + it , \chi \right) \right|^p \right)^{1/p} \]
and
\[ \| N \|_{\infty} = \max \left\{ \left| N \left( \frac{1}{2} + it , \chi \right) \right| : \chi \: \mbox{appears in} \; \sideset{}{^\dagger} \sum \right\} . \]
Lemma~\ref{Diripoly4powerest}, possibly in conjunction with a partial summation to incorporate a $\log n$ factor, gives that
\begin{equation} \label{Njest1}
\| N_j \|_4^4 \ll q D_1^2 \mathcal{L}^9 (1+|t|) \ll D_1 x^{1/2-2\varepsilon/3} (1+|t|)
\end{equation} 
if $N_j > x^{1/6}$.  If $N_j \leq x^{1/6}$, we obtain similar bounds from Lemma~\ref{larsie}, applied to $T=N_j^2$.  Indeed, in this case,
\[ \| N_j \|_4^4 \ll ( N_j^2 + q D_1^2 ) \mathcal{L}^4 \ll \left\{ \begin{array}{ll} D_1 x^{1/2-2\varepsilon/3} & \mbox{in all cases} \\ qD_1^2 \logl^4 & \mbox{if} \; \theta \geq 1/3 . \end{array} \right. \]

From now on, it is convenient to arrange $N_1, \cdots , N_{14}$ so that
\[ N_1 \geq \cdots \geq N_{14} . \]
The proof of \eqref{redugoal} is divided into three cases.

\subsection*{Case 1.}  Suppose that $N_1 N_2 \geq x_0^{1/2}$.  Let $M= N_3 \cdots N_{14}$.  Then the left-hand side of \eqref{redugoal} is
\[ \| M N_1 N_2 \|_1 \leq \| M \|_2 \| N_1 \|_4 \| N_2 \|_4 \ll (M + qD_1^2)^{1/2} D_1^{1/2} x^{1/4-\varepsilon/4} (1+ |t|)^{1/2} \ll x^{1/4} D_1 x^{1/4 - \varepsilon/4} (1+|t|)^{1/2}.  \]
by H\"{o}lder's inequality, Lemma~\ref{larsie} and \eqref{Njest1}.  So \eqref{redugoal} holds in Case 1.

\subsection*{Case 2.} $N_1 N_2 < x_0^{1/2}$ and some sub-product $\prod_{i \in \mathcal{S}} N_i$ (with $\mathcal{S} \subseteq \{ 1, \cdots, 14 \}$) satisfies
\begin{equation} \label{subprodbound}
 x_0^{1/2} \leq N = \prod_{i \in \mathcal{S}} N_i < x^{1-\theta-\varepsilon}. 
 \end{equation}
 Hence 
 \[ M = \prod_{\substack{ 1 \leq i \leq 14 \\  i \not\in \mathcal{S}}} N_i \leq x_0^{1/2}. \]
 The left-hand side of \eqref{redugoal} is, using Lemma~\ref{larsie} and with $C_4$ suitably chosen,
 \[ \| M N \|_1 \leq \| M \|_2 \| N \|_2 \ll (M+qD_1^2)^{1/2} (N+qD_1^2)^{1/2} \mathcal{L}^{C_4/2} \ll (x_0^{1/2} + qD_1^2 + N^{1/2}q^{1/2} D_1) \mathcal{L}^{C_4/2} . \]
We clearly have $x_0^{1/2} \mathcal{L}^{C_4/2} \ll x^{1/2} D_1 \mathcal{L}^{-2A-20}$ as $D_1 \geq \mathcal{L}^{2A+C_4}$ and $qD_1^2 \mathcal{L}^{C_4/2} \ll x^{1/2} D_1 \mathcal{L}^{-2A-20}$.  Lastly, using \eqref{subprodbound},
\[ N^{1/2} q^{1/2} D_1 \mathcal{L}^{C_4/2} \ll x^{1/2-\varepsilon/2} D_1 \mathcal{L}^{C_4/2} \ll x^{1/2} D_1 \mathcal{L}^{-2A-20}. \]
So \eqref{redugoal} also holds in Case 2. \newline

We claim that if $\theta \leq 2/5-\varepsilon$, then Case 1 or Case 2 must occur.  Suppose not, then
\[ \alpha_1 + \alpha_2 < \frac{1}{2} \]
and there is no sub-sum with
\[ \frac{2}{5} \leq \sum_{i \in \mathcal{S}} \alpha_i \leq \frac{3}{5}. \]
One can easily verify that this is impossible.  See the details in Lemma 14 of \cite{Bak}. \newline

Now we suppose that $2/5-\varepsilon < \theta < 5/12$ and it still remains to consider

\subsection*{Case 3.} $N_1 N_2 < x_0^{1/2}$ and no sub-product $\prod_{i \in \mathcal{S}} N_i$ satisfies \eqref{subprodbound}.  Since $1-\theta - \varepsilon \geq 7/12$ (take $\varepsilon < \eta$), no sub-product $\prod_{i \in \mathcal{S}} N_i$ lies in $[ x_0^{5/12}, x_0^{7/12} ]$.  We start with a combinatorial lemma.

\begin{lemma} \label{comblem}
Suppose that $\alpha_1 \geq \cdots \geq \alpha_{14} \geq 0$, $\alpha_1 + \cdots + \alpha_{14} =1$, $\alpha_1 + \alpha_2 < 1/2$ and no sub-sum $\sum_{i \in \mathcal{S}} \alpha_i$ for a set $\mathcal{S} \subset \{ 1, \cdots , 14 \}$ is in $[5/12, 7/12]$.  Then $\alpha_5 > 1/6$ and
\begin{equation} \label{sumlistbound}
\alpha_1 + \alpha_2 + \alpha_6 + \alpha_7 + \cdots + \alpha_{14} < \frac{5}{12}.
\end{equation}
\end{lemma}

\begin{proof}
Clearly $\alpha_1 + \alpha_2 < 5/12$.  Suppose that
\begin{equation} \label{missumbound}
\alpha_1 + \alpha_2 + \sum_{\alpha_i \leq 1/6} \alpha_i \geq \frac{5}{12} . 
\end{equation}
Let $s$ be the least sum $\alpha_1 + \alpha_2 + \sum_{i \in \mathcal{B}} \alpha_i$, for some set $\mathcal{B} \subset \{ i: \alpha_i \leq 1/6 \}$, that is greater than $5/12$.  This implies that $5/12 \leq s < 5/12+1/6=7/12$, contradicting one of the conditions of the lemma.  So \eqref{missumbound} must be false. \newline

We can write $\{ i : \alpha_i \leq 1/6 \}$ as $\{ i : i > t \}$ for some $t$ with $1 \leq t \leq 14$, since the $\alpha_i$'s are in descending order.  If $t \geq 6$, then 
\[ \alpha_1 + \cdots + \alpha_{14} \geq \alpha_1 + \cdots + \alpha_t > \frac{t}{6} \geq 1 \]
which is false.  If $t \leq 4$, then
\[ \alpha_1 + \cdots + \alpha_{14} \leq \left( \alpha_1 + \alpha_2 + \sum_{i > t} \alpha_i \right) + \left( \alpha_3 + \alpha_4 \right) < \frac{5}{12} + \frac{5}{12} < 1 \]
which is also false.  Therefore, $t=5$ and both claims of the lemma are proved.
\end{proof}

By Lemma~\ref{comblem}, in Case 3, we can partition $N_1 \cdots N_{14}$ into three parts $M$, $N$ and $N_5$,
\[ M(s, \chi) = N_1(s,\chi) N_2(s, \chi) \prod_{i \geq 6} N_i (s,\chi) = \sum_{M \leq m \ll M} \alpha_m \chi(m) m^{-s} \]
and
\[ N(s,\chi) = N_3 (s, \chi) N_4(s,\chi) = \sum_{N \leq n \ll N} \beta_n \chi(n) n^{-s} , \]
where $M < x_0^{5/12}$, $N < x_0^{5/12}$, $N_5 > x_0^{1/6}$, $M \geq N$.  So $MN_5 \geq x_0^{1/2}$. \newline

We need the stronger assertion that
\begin{equation} \label{N5bound}
 N_5 > x^{1/6-\varepsilon} .
\end{equation}
If this does not hold, then
\[ x_0^{1/2} \leq M N_5 < x_0^{5/12} x^{1/6-\varepsilon} < x^{1-\theta-\varepsilon} , \]
an impossibility in Case 3.  \newline

The utility of \eqref{N5bound} stems partly from the following lemma.
\begin{lemma} \label{burgess}
Let $\chi$ be a character modulo $q_1 d$ that appears in $\sum^{\dagger}$.  Then
\[ \sum_{k \leq K} \chi (k) \ll K^{1-\varepsilon/2} \]
whenever $K \geq x^{3/20}$.
\end{lemma}

\begin{proof}
By a theorem of D. A. Burgess \cite{Bur}, we have
\[ \sum_{k \leq K} \chi(n) \ll (q_1 d)^{1/9 + \varepsilon^2} K^{2/3} \ll x^{(9/20 - \varepsilon)(1/9+\varepsilon^2)} K^{2/3} \ll K^{1-\varepsilon/2} , \]
completing the proof.
\end{proof}

The coefficients in $N_5$ cannot involve the M\"obius $\mu$-function.  Otherwise, from Lemma~\ref{HBID} and \eqref{N5bound}, we get
\[ x^{1/7} \geq 2N_5 > 2x^{1/6-\varepsilon} ,\]
which is false if $\varepsilon$ is sufficiently small.  Now it is easy to obtain
\begin{equation} \label{N5infnormest}
\| N_5 \|_{\infty} \ll N_5^{1/2} x^{-\varepsilon/13} (1+|t|)
\end{equation}
from Lemma~\ref{burgess}, \eqref{N5bound} and a partial summation argument.  \newline

The contribution in \eqref{redugoal} from $\chi$ with
\[ \min \left\{ \left| M \left( \frac{1}{2} + it , \chi \right) \right|, \left| N \left( \frac{1}{2} + it  , \chi \right) \right|,  \left| N_5 \left( \frac{1}{2} + it , \chi \right) \right| \right\} < x^{-1} \]
is clearly
\[ \ll \sideset{}{^\dagger} \sum 1 \ll x^{1/2 -\varepsilon} D_1 . \]
Therefore, by a splitting-up argument, it suffices to show, for any $U$, $V$ and $W$ with
\[ U \leq \| N_5 \|_{\infty}, \; \; V \leq \| M \|_{\infty} \; \; \mbox{and} \; \; W \leq \| N \|_{\infty} , \]
that
\[ UVW \# A(U,V,W) \ll (1 + |t|) x^{1/2} D_1 \mathcal{L}^{-2A-20} . \]
Here
\begin{equation*}
\begin{split}
 A(U,V,W) = \Bigg\{ \chi : \chi \; & \mbox{appears in} \;  \sideset{}{^\dagger}\sum, U < \left| N_5 \left( \frac{1}{2} + it, \chi \right) \right| \leq 2U,  \\
 & V < \left|  M \left( \frac{1}{2} + it, \chi \right) \right| \leq 2V, \; \;  W < \left| N \left( \frac{1}{2} + it, \chi \right) \right| \leq 2W \Bigg\} .
 \end{split}
 \end{equation*}
Now let 
\[ P = \min \left\{ \frac{M+qD_1^2}{V^2}, \frac{N+qD_1^2}{W^2} , \frac{qD_1^2}{U^4} (1+|t|), \frac{M}{V^2} + \frac{qD_1^2M}{V^6}, \frac{N}{W^2}+\frac{qD_1^2N}{W^6}, \frac{N_5^2}{U^4} + \frac{qD_1^2N_5^2}{U^{12}} \right\}. \]
It is a consequence of Lemmas~\ref{larsie}, \ref{huxlem} and the first inequality in \eqref{Njest1} that
\[  \# A(U,V,W) \ll P x^{\varepsilon/20} . \]
So it is enough to show that
\begin{equation} \label{newgoal}
UVWP \ll x^{1/2 - \varepsilon/13} D_1 (1+|t|) . 
\end{equation}
To do this, we consider four sub-cases, according to the size of $P$ in comparison with those of $2V^{-2}M$ and $2W^{-2} N$. \newline

\subsection*{(a) $P \leq 2 V^{-2} M$ and $P \leq 2W^{-2}N$} In this case, \eqref{N5infnormest} yields
\[ UVWP \ll UVW (V^{-2}M)^{1/2} (W^{-2}N)^{1/2} \ll (MN)^{1/2} \| N_5 \|_{\infty} \ll x^{1/2 - \varepsilon/13} (1+|t|) , \]
as desired for \eqref{newgoal}.

\subsection*{(b) $P > 2 V^{-2} M$ and $P > 2W^{-2}N$} Here, we have
\begin{eqnarray*}
P & \leq & 2 \min \bigg\{ qD_1^2 V^{-2}, qD_1^2 W^{-2}, qD_1^2MV^{-6}, qD_1^2 NW^{-6}, (1+|t|) qD_1^2 U^{-4}, N_5^2 U^{-4} \bigg\} \\
  & & + 2 \min \bigg\{ qD_1^2 V^{-2}, qD_1^2 W^{-2}, qD_1^2MV^{-6}, qD_1^2 NW^{-6}, (1+|t|) qD_1^2 U^{-4}, qD_1^2N_5^2 U^{-12} \bigg\} \\
  & \leq & 2(qD_1^2 V^{-2})^{5/16} (qD_1^2W^{-2})^{5/16} (qD_1^2MV^{-6})^{1/16} (qD_1^2NW^{-6})^{1/16} \left( \min \{ qD_1^2 U^{-4}, N_5^2 U^{-4} \} \right)^{1/4} (1+|t|)^{1/4} \\
  & & + 2 \min \bigg\{ (qD_1^2 V^{-2})^{5/16} (qD_1^2W^{-2})^{5/16} (qD_1^2MV^{-6})^{1/16} (qD_1^2NW^{-6})^{1/16} (qD_1^2U^{-4})^{1/4} (1+|t|)^{1/4}, \\
 & &  \hspace*{0.8in} (qD_1^2V^{-2})^{7/16} (qD_1^2 W^{-2})^{7/16} (qD_1^2MV^{-6})^{1/48} (qD_1^2NW^{-6})^{1/48} (qD_1^2N_5^2U^{-12})^{1/12} \bigg\} \\
  & \leq & 2 (1+|t|)^{1/4} (UVW)^{-1} qD_1^2 (MN)^{1/16} \left( \min \left\{ 1, (qD_1^2)^{-1/4} N_5^{1/2} \right\} + \min \left\{ 1, N_5^{1/6} (MN)^{-1/24} \right\} \right) \\
  & \ll & (1+|t|)^{1/4} (UVW)^{-1} \left( x^{1/16} (qD_1^2)^{31/32} + x^{1/20} qD_1^2 \right). 
\end{eqnarray*}
Now, noting that
\[ x^{1/16} (qD_1^2)^{31/32} \ll x^{1/16+31/32 \cdot 9/20} D_1^{31/32} \ll x^{1/2 - \varepsilon} D_1 \]
and
\[ x^{1/20} qD_1^2 \ll x^{1/20} x^{9/20-\varepsilon} D_1 \ll x^{1/2-\varepsilon} D_1, \]
we get that
\[ P \ll (1+|t|)^{1/4} (UVW)^{-1} x^{1/2-\varepsilon} D_1 , \]
which gives \eqref{newgoal}.

\subsection*{(c) $P > 2 V^{-2} M$ and $P \leq 2W^{-2}N$} Now we have
\begin{eqnarray*}
P & \leq & 2 \min \bigg\{ qD_1^2 V^{-2}, NW^{-2}, qD_1^2 MV^{-6}, qD_1^2 U^{-4} (1+|t|), N_5^2 U^{-4} \bigg\} \\
 & & + 2 \min \bigg\{ qD_1^2 V^{-2}, NW^{-2}, qD_1^2 MV^{-6}, qD_1^2 U^{-4} (1+|t|) , qD_1^2N_5^2 U^{-12} \bigg\} \\
 & \leq & 2 (qD_1^2V^{-2})^{1/8} (NW^{-2})^{1/2} (qD_1^2MV^{-6})^{1/8} \left( \min \left\{ qD_1^2 U^{-4}, N_5^2 U^{-4} \right\} \right)^{1/4} (1+|t|)^{1/4} \\
 & & + 2 \min \bigg\{ (qD_1^2V^{-2})^{1/8} (NW^{-2})^{1/2} (qD_1^2MV^{-6})^{1/8} (qD_1^2U^{-4})^{1/4} (1+|t|)^{1/4} , \\
 & & \hspace*{1in} (qD_1^2V^{-2})^{3/8} (NW^{-2})^{1/2} (qD_1^2MV^{-6})^{1/24} (qD_1^2N_5^2U^{-12})^{1/12} \bigg\} \\
 & \leq & 2 (1+|t|)^{1/4} (UVW)^{-1} (qD_1^2N)^{1/2} M^{1/8} \left( \min \bigg\{ 1, (qD_1^2)^{-1/4}N_5^{1/2} \bigg\} + \min \bigg\{ 1, N_5^{1/6}M^{-1/12} \bigg\} \right) \\
 & \ll & (1+|t|)^{1/4} (UVW)^{-1} \left( x^{1/8} (qD_1^2)^{7/16} N^{3/8} + x^{1/12} (qD_1^2)^{1/2} N^{5/12} \right) .
\end{eqnarray*}
To estimate these last two terms, we have
\[ x^{1/8}  (qD_1^2)^{7/16} N^{3/8} \ll x^{1/8}  (qD_1)^{7/16} D_1^{7/16} (x^{5/12})^{3/8} \ll x^{1/8+9/20 \cdot 7/16 + 5/12 \cdot 3/8} D_1^{7/16} \ll x^{1/2-\varepsilon} D_1  \]
and
\[ x^{1/12} (qD_1^2)^{1/2} N^{5/12} \ll x^{1/12} (qD_1)^{1/2} D_1^{1/2} x^{25/144} \ll x^{1/12 + 9/40 + 25/144} D_1^{1/2} \ll x^{1/2-\varepsilon} D_1 . \]
These bounds lead to
\[ P \ll (1+|t|)^{1/4} (UVW)^{-1} x^{1/2-\varepsilon} D_1 , \]
giving \eqref{newgoal}.

\subsection*{(d) $P > 2 W^{-2} N$ and $P \leq 2V^{-2}M$}  We proceed the same way as in subcase (c), interchanging the roles of $M$ and $N$. \newline

This completes the proof of Theorem~\ref{bomvino}.

\section{Proof of Theorem~\ref{bardavhal}}

From (iii) of Lemma~\ref{charredu}, we get
\[ \sum_{\substack{d \leq Q/q \\ (d,q)=1}} \sum_{\substack{a=1 \\ (a,dq)=1}}^{dq} \left( \psi(x; dq, a) - \frac{x}{\varphi(qd)} \right)^2 \ll \sum_{\substack{d \leq Q/q \\ (d,q)=1}} \frac{1}{\varphi(qd)} \sideset{}{^{\prime}} \sum_{\chi \bmod{dq}} \left| \psi (x, \chi) \right|^2 + \frac{x^2}{\varphi(q) \logl^{2A}} \sum_{d \leq Q/q} \frac{1}{\varphi(d)} . \]
As the second term is $ \ll Qx \varphi(q)^{-1}$, it suffices to prove that
\begin{equation} \label{BDHtrans1}
\sum_{\substack{d \leq Q/q \\ (d,q)=1}} \frac{1}{\varphi(qd)}  \sideset{}{^{\prime}} \sum_{\chi \bmod{dq}} \left| \psi ( x, \hat{\chi} ) \right|^2 \ll \frac{Q x \mathcal{L}}{\varphi(q)} 
\end{equation}
and that
\begin{equation} \label{BDHtrans2}
\sum_{\substack{d \leq Q/q \\ (d,q)=1}} \frac{1}{\varphi(qd)} \sideset{}{^{\prime}} \sum_{\chi \bmod{dq}} \left( \left| \psi(x, \hat{\chi}) \right|^2 - \left| \psi(x, \chi) \right|^2 \right) \ll \frac{Qx \logl}{\varphi(q)} .
\end{equation}
It is easy to see that, in \eqref{BDHtrans2},
\[ \left| \psi(x, \hat{\chi}) \right|^2 - \left| \psi(x, \chi) \right|^2 \ll \left( \sum_{p^k \leq x} \log p \right) \left( \sum_{p | dq} \log p \right) . \]
The contribution to \eqref{BDHtrans2} from $k \geq 2$ is
\[ \ll \sum_{\substack{d \leq Q/q \\ (d,q)=1}} x^{1/2+\varepsilon} \ll \frac{Qx^{1/2+\varepsilon}}{q} \]
which is acceptable.  The contribution from $k=1$ to \eqref{BDHtrans2} is
\[ \ll \sum_{\substack{ d \leq Q/q \\ (d,q)=1}} x \sum_{\substack{p \leq x \\ p | dq}} \log p \ll x \sum_{\substack{d \leq Q/q \\ (d,q)=1}} \sum_{p|q} \log p + x \sum_{p \leq x} \log p \sum_{\substack{d \leq Q/q \\ d \equiv 0 \bmod{p}}} 1 \ll \frac{xQ}{q} \log q + \frac{xQ}{q} \sum_{p \leq x} \frac{\log p}{p} \ll \frac{xQ \logl}{\varphi(q)}   \]
which is also acceptable.  (Incidentally, the error term corresponding to \eqref{BDHtrans2} is treated incorrectly on page 170 of \cite{HD}; the above discussion corrects this minor error.) \newline

It remains to prove \eqref{BDHtrans1} in the form
\begin{equation} \label{BDHtrans}
\sum_{q_1 | q} \sum_{\substack{d_1 \leq Q/q \\ (d_1, q)=1}} \sum_{\substack{d \leq Q/q \\ d_1 | d \\ (d,q)=1}} \frac{1}{\varphi(d)} \sideset{}{^{\star}} \sum_{\chi \bmod{d_1q_1}} \left| \psi( x, \chi) \right|^2 \ll Q x \logl .
\end{equation}
We split the sum over $d_1$ in \eqref{BDHtrans} into dyadic sub-sums of the form $\sum_{D < d_1 \leq 2D}$ where $D$ takes on the values $2^{-k} Q/q$, $k \geq 1$ and $2^{-k} Q/q > 1/2$.  Let $\Sigma_D$ denote the contribution to \eqref{BDHtrans} from a given $D$.  Hence
\[ \Sigma_D \ll \left( \log \frac{Q}{qD} \right) \sum_{\substack{ D < d_1 \leq 2D \\ (d_1,q)=1}} \frac{1}{\varphi(d_1)} \sum_{q_1 | q} \ \ \sideset{}{^{\star}} \sum_{\chi \bmod{d_1q_1}} \left| \psi ( x, \chi ) \right|^2 . \]
We first deal with the contributions from $D \leq \mathcal{L}^{2A}$:
\begin{equation*}
\begin{split}
 \sum_{D\leq \mathcal{L}^{2A}} \Sigma_D & \ll \mathcal{L} x \sum_{\substack{d_1 \leq 2\mathcal{L}^{2A} \\ (d_1, q) =1}} \frac{1}{\varphi(d_1)} \sum_{q_1 | q} \ \sideset{}{^{\star}} \sum_{\chi \bmod{d_1q_1}} \left| \psi ( x, \chi ) \right| \\
 & \ll \mathcal{L} x \sum_{\substack{d \leq 2\mathcal{L}^{2A} \\ (d,q)=1}} \frac{1}{\varphi(d)} \sideset{}{^{\prime}} \sum_{\chi \bmod{dq}} \left( |\psi (x, \chi)| + \logl^2 \right) \ll \frac{x^2}{\logl^A} + x \varphi(q) \logl^{2A+3} \ll Q x \logl,
 \end{split}
 \end{equation*}
where we have used (ii) of Lemma~\ref{charredu} and estimates occurring in the proof of \eqref{chidagineq}. \newline

Now for the remaining $D$'s with $D > \mathcal{L}^{2A}$, we use Lemma~\ref{larsie} and get
\[ \Sigma_D \ll \frac{1}{D} \log \frac{Q}{qD} \left( x + qD^2 \right) \sum_{n \leq x} \Lambda^2 (n) \ll \frac{x \mathcal{L}}{D}  \log \frac{Q}{qD} \left( x + qD^2 \right) . \]
Now we observe easily that
\[ \sum_{D > \mathcal{L}^{2A}} \frac{x^2 \mathcal{L}}{D} \log \frac{Q}{qD} \ll \frac{x^2}{\mathcal{L}^A} \ll Qx \]
and
\[ \sum_{D > \mathcal{L}^{2A}} qx \mathcal{L} D \log \frac{Q}{qD} \ll qx \mathcal{L} \sum_{k \geq 1} k \frac{Q}{q 2^k} \ll Qx \mathcal{L}. \]
This completes the proof of Theorem~\ref{bardavhal}.

\section{Proof of Theorem~\ref{maintheo}}

We say that a set $\mathcal{H} = \{ h_1, \cdots, h_k \}$ of distinct non-negative integers is {\it admissible} if for every prime $p$, there is an integer $a_p$ such that
\[ a_p \not\equiv h \pmod{p} \]
for all $h \in \mathcal{H}$. \newline

For a set of natural numbers $\mathcal{A}$, we write $X (\mathcal{A}; n)$ for the indicator function of $\mathcal{A}$.  For a smooth function $F$ supported on
\[ \mathcal{R}_k = \{ (x_1, \cdots, x_k) \in [0,1]^k : \sum_{i=1}^k x_i \leq 1 \} \]
and $1 \leq m \leq k$, let
\[ I_k (F) = \int\limits_0^1 \cdots \int\limits_0^1 F( t_1, \cdots, t_k ) ^2 \dif t_1 \cdots \dif t_k \]
and
\[ J^{(m)}_k (F) = \int\limits_0^1 \cdots \int\limits_0^1 \left( \int\limits_0^1 F( t_1, \cdots, t_k ) \dif t_m \right)^2 \dif t_1 \cdots \dif t_{m-1} \dif t_{m+1} \cdots \dif t_k . \]
Furthermore, set
\[ M_k = \sup_F \frac{\sum_{m=1}^k J_k^{(m)}(F)}{I_k (F)} , \]
where the supremum is taken over $F$ described above with $I_k (F) \neq 0$, $J^{(m)}_k (F) \neq 0$ for $m = 1, \cdots , k$.  It is shown in \cite{May} that
\[ M_k \geq \log k - 2 \log \log k + O(1) . \]
This bound is strengthened slightly in \cite{polym} to
\begin{equation} \label{Mklowerbd}
M_k \geq \log k + O(1) . 
\end{equation}

We now state a special case of \cite[Theorem 1]{BakZha1} for the integers $q$ and $a$ in the introduction.  Set
\[ D_0 = \frac{\log \log (x/2)}{\log \log \log (x/2)}. \]

\begin{lemma} \label{specmaylem}
Let $t$, $k$ be natural numbers and $L$ be a positive constant such that
\[ M_k > \frac{2t-2}{L}. \]
Let $\mathcal{H} = \{ h_1, \cdots, h_k \}$ be an admissible set with $h_1 < \cdots < h_k$, with $q| h_j $ for $j = 1, \cdots, k$.  Suppose that $p | h_i - h_j$ with $i \neq j$, $p > D_0$ implies $p | q$.  Let $x$ be large in terms of $k$ and
\[ \mathcal{A} = \left\{ n : \frac{x}{2} < n \leq x, n \equiv a \pmod{q} \right\} \; \; \; \mbox{and} \; \; \;\mathbb{P} = \{ p : p \in \mathcal{A} \} . \]
Set 
\[ Y = \frac{x}{2 q} \; \; \; \mbox{and} \; \; \; Y_1 = \frac{1}{\varphi(q)} \int\limits_{x/2}^x \frac{\dif t}{\log t} . \]
Suppose that
\begin{equation} \label{maycond1}
\sum_{\substack{d \leq x^L \\ (d,q)=1}} \mu^2 (d) \tau_{3k} (d) \left| \sum_{n \equiv b_d \bmod{qd}} X (\mathcal{A}; n) - \frac{Y}{d} \right| \ll \frac{Y}{\mathcal{L}^{k+\varepsilon}}
\end{equation}
for any $b_d \equiv a \pmod{q}$, and
\begin{equation} \label{maycond2}
\sum_{\substack{d \leq x^L \\ (d,q)=1}} \mu^2 (d) \tau_{3k} (d) \left| \sum_{n \equiv b_d \bmod{qd}} X ((\mathcal{A}+h_m) \cap \mathbb{P}; n) - \frac{Y_1}{\varphi(d)} \right| \ll \frac{Y}{\mathcal{L}^A}
\end{equation}
for every integer $b_d \equiv a \pmod{q}$ with $(b_d, q) =1$.  Then there are primes $p_1 < \cdots < p_t$ in $\mathcal{A}$ satisfying
\[ p_t - p_1 \leq h_k - h_1. \]
\end{lemma}

\begin{proof} [Proof of Theorem~\ref{maintheo}]
We may suppose that $t$ is sufficiently large.  Suppose that $q$ satisfies \eqref{thetaccond}.  Let
\[ \mathcal{A} = \left\{ n \in \left( \frac{x}{2}, x \right] : n \equiv a \pmod{q} \right\}  \]
and $0 \leq h'_1 < \cdots < h'_k$ be an admissible set with
\[ h'_k \ll k \log k . \]
Then $\mathcal{H} = \{ h'_1 q , \cdots , h'_k q \}$ is an admissible set for which $p > D_0$, $p | h_i - h_j (i \neq j)$ implies $p | q$.  Further,
\[  h'_kq - h'_1 q \ll q k \log k . \]
Here we choose the least $k$ such that
\[ M_k > \frac{2t-2}{L(\theta) +\varepsilon/2} . \]
Mindful of \eqref{Mklowerbd}, we get
\[ \log k  \leq \frac{2t}{L ( \theta) + \varepsilon/2} + O(1) . \]
Choosing $\varepsilon$ sufficiently small, and recalling that $t$ is large, we have
\[ (h_k'-h_1') q \ll q \exp \left( \frac{2t}{L(\theta)} \right) . \]
It now remains to verify that the hypotheses of Lemma~\ref{specmaylem} are satisfied with $L = L (\theta)$. \newline

The bound \eqref{maycond1} presents no difficulty, as 
\[ \sum_{n \equiv b_d \bmod{dq}} X (\mathcal{A}; n) = \frac{Y}{d} + O(1) . \]
To verify \eqref{maycond2}, we observe that for $(d,q)=1$, $b_d \equiv a \pmod{q}$ and $(b_d, dq) =1$, 
\[ \sum_{n \equiv b_d \bmod{dq}} X ((\mathcal{A}+h_m) \cap \mathbb{P}; n) = \sum_{\substack{p \equiv b_d \bmod{dq} \\ x/2 + h_m < p \leq x }} 1  . \]

Let $A=10k^2$ and
\[ R_d = \left| \sum_{\substack{ p \equiv b_d \bmod{dq} \\ x/2+h_m < p \leq x}}  1 - \frac{Y_1}{\varphi(d)} \right| . \]
Let $L=L(\theta) + \varepsilon/2$.  We readily deduce from Theorem~\ref{bomvino}, with $\varepsilon/2$ in place of $\varepsilon$, that
\[ \sum_{\substack{d \leq x^L \\ (d,q)=1}} R_d \ll \frac{Y}{\logl^A}; \]
compare the argument at the end of \cite{El2}.  Hence the Cauchy-Schwarz inequality together with the Brun-Titchmarsh inequality gives
\begin{equation*}
\begin{split}
 \sum_{\substack{d \leq x^L \\ (d,q)=1}} \mu^2(d) \tau_{3k} (d) & R_d \leq \left( \sum_{\substack{d \leq x^L \\ (d,q)=1}} \mu^2(d) \tau_{3k}^2 (d) R_d \right)^{1/2} \left( \sum_{\substack{d \leq x^L \\ (d,q)=1}} R_d \right)^{1/2} \\
 & \ll Y \left( \sum_{d \leq x^L} \frac{\tau^2_{3k} (d)}{\varphi(d)} \right)^{1/2} \logl^{-A/2} \ll Y \logl^{(9k^2-A)/2} \ll Y \logl^{-(k+\varepsilon)} .
 \end{split}
 \end{equation*}

Now we may apply Lemma~\ref{specmaylem} and obtain primes $p_1 < \cdots < p_t$ in $\mathcal{A}$ with
\[ p_t - p_1 \leq (h'_k - h'_1)q \ll q \exp \left( \frac{2t}{L(\theta)} \right) .\]
This completes the proof of Theorem~\ref{maintheo}.
\end{proof}

\vspace*{.5cm}

\noindent{\bf Acknowledgments.}  We would like to thank Tristan Freiberg for pointing out a significant defect in a previous version of this paper.  This work was done while L. Z. held a visiting position in the Department of Mathematics of Brigham Young University (BYU).    He wishes to thank the warm hospitality of BYU during his thoroughly enjoyable stay in Provo.

\bibliography{biblio}

\vspace*{.5cm}

\noindent\begin{tabular}{p{8cm}p{8cm}}
Roger C. Baker & Liangyi Zhao \\
Department of Mathematics & School of Mathematics and Statistics \\
Brigham Young University & University of New South Wales \\
Provo, UT 84602 USA & Sydney, NSW 2052 Australia \\
Email: {\tt baker@math.byu.edu} & Email: {\tt l.zhao@unsw.edu.au} \\
\end{tabular}

\end{document}